\documentclass{amsart}
\usepackage{amsmath}
\usepackage{amssymb}
\usepackage{amsthm}
\usepackage{amsfonts}
\usepackage{url}
\usepackage{hyperref}
\usepackage{appendix}
\usepackage{verbatim}
\usepackage{graphicx,color}
\usepackage{lineno}

\def\<{\langle}
\def\>{\rangle}

\newtheorem{thmA}{Theorem}

\title{Profinite completions of free-by-free groups contain everything}
\author{Martin R. Bridson} 
\address{Mathematical Institute,  Andrew Wiles Building,
University of Oxford, OX2 6GG, UK}
\email{ bridson@maths.ox.ac.uk}

\def\-{\overline}

\def\G{\Gamma}

\def\ker{{\rm{ker}}\ }
\def\wh{\widehat}

 \def\Z{\mathbb{Z}}

\def\qed{ $\sqcup\!\!\!\!\sqcap$}

\def\G{\Gamma}

\def\<{\langle}
\def\>{\rangle}
\def\ilim{\varprojlim}

\def\out{{\rm{Out}}}

\def\G{\Gamma}
\def\wh{\widehat}

\newtheorem{theorem}{Theorem}[section]
\newtheorem{lemma}[theorem]{Lemma}

\newtheorem{proposition}[theorem]{Proposition}

\def\out3{{\rm{Out}}(F_3)}

\theoremstyle{definition} 

\newtheorem{remark}[theorem]{Remark}

\newtheorem*{theorem*}{Theorem}

\def\ilim{\varprojlim}

\begin{document}
  
\begin{abstract}   
Given an arbitrary, finitely presented, residually finite group $\G$, one can construct a finitely generated,
residually finite, 
free-by-free group $M_\Gamma = F_\infty\rtimes F_4$ and an embedding
$M_\Gamma  \hookrightarrow (F_4\ast \G)\times F_4$
that induces an isomorphism of profinite completions. In particular, there is a free-by-free
group whose profinite completion contains $\wh{\G}$ as a retract.
\end{abstract}




\maketitle

%
%
%
%


The finite quotients of a  group $\G$ form a directed system
and  the  {\em profinite completion} of $\G$ is the inverse limit of this system,
$
\wh{\G}:= \ilim \G/N.$
 The natural map $\G\to\wh{\G}$ is injective if $\G$ is residually finite, and two finitely generated
groups $\G_1$ and $\G_2$ have the same set of finite images if and only if $\wh{\G}_1\cong \wh{\G}_2$.
The purpose of this note is to demonstrate that the profinite completions of finitely generated,
residually finite free-by-free groups 
contain, as retracts, the profinite completions of all subgroups of finitely presented groups.

\begin{thmA} \label{t:B}
Given an arbitrary, finitely generated, recursively presented  group $\G$
that is residually finite, one can construct a finitely generated, residually finite
free-by-free group $M_\Gamma = F_\infty\rtimes F_4$ and an embedding
$M_\Gamma  \hookrightarrow (F_4\ast \G)\times F_4$
that induces an isomorphism of profinite completions.
\end{thmA}
Note that $D(\G):=(F_4\ast \G)\times F_4$ is residually finite
and the obvious retraction $D(\G)\to\G$ induces a retraction  $\wh{M}_\Gamma\cong\wh{D(\G)}\to\wh{\G}$.

It follows from this theorem
that  the cohomological dimension of the profinite completion of a  finitely generated, residually finite group of
cohomological dimension $2$ can be any positive integer, or can be infinite (Section \ref{s:cohom}).  
 And, despite
being torsion-free itself, a free-by-free group can have all manner of torsion in its profinite completion (Section \ref{s:torsion}). The theorem also tells us
that, with the possible (but unlikely) exception of certain free-by-free groups $H$, no statement of the following
form can be valid for all pairs of finitely generated, residually finite groups $\G_1$ and $\G_2$:
 ``{\em{if $\wh{\G}_1\cong\wh{\G}_2$ and $\G_1$ has a subgroup isomorphic to $H$, then
$\G_2$ has a subgroup isomorphic to $H$.}}"
Furthermore, the theorem tells that if a property
$\mathcal P$ is common to the subgroups of free-by-free groups but not to the subgroups of all
finitely presented, residually finite groups, then $\mathcal P$ is {\em not}  a profinite invariant. Such
properties include: being torsion-free; being locally indicable (i.e.~every finitely generated subgroup maps onto $\Z$); 
being left-orderable;
all 2-generator subgroups being finitely presented (or coherent); 
 all solvable (or amenable, or nilpotent) subgroups being finitely generated and abelian (of rank at most $2$).

We shall see that Theorem \ref{t:B} is a  rather direct consequence of the following result,   which is
proved in \cite{mb-gilb} using celebrated theorems of Higman \cite{higman}
and Baumslag, Dyer and Heller \cite{BDH}.
A group $G$ is termed {\em acyclic} if $H_i(G,\Z)=0$ for all $i\ge 1$. 

\begin{thmA}\cite{mb-gilb} \label{t:A}
There is a finitely presented acyclic group $U$ such that:
\begin{enumerate}
\item $U$ has no proper subgroups of finite index;
\item every finitely generated, recursively presented group can be embedded in $U$.
\end{enumerate}
\end{thmA}
 
\section{The Construction}

The {\em fibre product} of a pair of epimorphisms $p_i:G_i\twoheadrightarrow Q \ (i=1,2)$ is the subgroup
$P=\{(g_1,g_2)\mid p_1(g_1)=p_2(g_2)\} < G_1\times G_2$.
We need the following well-known lemma.

\begin{lemma}\label{ll} If $G_1$ and $G_2$ are finitely generated and $Q$ is finitely
presented, then $P$ is finitely generated.
\end{lemma}

\begin{proof} For $i=1,2$, let $S_i\subset G_i$ be a finite set that generates $G_i$.
For each $s\in S_1$ choose  $u_s\in G_2$ so that $p_1(s) = p_2(u_s)$. Similarly, 
for each $t\in S_2$ choose $v_t\in G_1$ so that $p_2(t)=p_1(v_t)$.
As $G_1$ is finitely generated
and $Q$ is finitely presented, there is a finite set $R\subset G_1$ 
whose conjugates  generate $\ker p_1$. It is easy to check
 $P$ is generated by $
\{ (s, u_s),\ (v_t,t),\ (r, 1) \mid r\in R,\ s\in S_1,\ t\in S_2\}.$
\end{proof}

The following proposition originates in the work of Platonov and Tavgen \cite{PT}. They only considered the
case $p_1=p_2$ but the adaptation to the asymmetric case is straightforward  \cite{mb-jems}.

\begin{proposition}\label{p:PT} For $i=1,2$, let $p_i:G_i\twoheadrightarrow Q$ be an epimorphism of groups.
If $G_1$ and $G_2$ are finitely generated and $Q$ is finitely presented, with $\wh{Q}=1$ and $H_2(Q,\Z)=0$,
then the inclusion of the fibre product $P\hookrightarrow G_1\times G_2$ induces an isomorphism
of profinite completions $\wh{P}\cong \wh{G}_1\times\wh{G}_2$.
\end{proposition} 

We shall need the following refinement of Theorem \ref{t:A}.
I do not contend that there
is any real significance to the integer $4$ in this statement (and Theorem \ref{t:B}); 
with sufficient effort one might well be able to construct a 2-generator group $U$ with the desired properties. 

\begin{lemma}
There is a 4-generator group $U$ with the properties described in Theorem \ref{t:A}.
\end{lemma}

\begin{proof} Theorem \ref{t:A} is proved on pages 20-21 of \cite{mb-gilb}. The construction of $U$ begins with
Higman's universal group $U_0$, which can be generated by $2$ elements.
A particular HNN extension $U^\dagger$ of $U_0$ is constructed and $U$ is an amalgamated free product $U^\dagger\ast_\Z B$
where $B$ is any
 finitely presented acyclic group  that has an element of infinite order $\tau$ but no non-trivial finite quotients.
The amalgamation identifies $\tau$ with the stable letter of the HNN extension $U^\dagger$, so 
$U$ is generated by $U_0$ and $B$.
We take $B$ to be the 2-generator group constructed in \cite{OS}.
\end{proof}

\noindent{\bf Proof of Theorem \ref{t:B}.} Let $U$ be a 4-generator group that
satisfies Theorem \ref{t:A}.  We fix an epimorphism $\mu: F_4\to U$.
Given a finitely generated, recursively presented group $\G$, we fix an embedding $\psi:\G\hookrightarrow U$
and extend this to an epimorphism $\Psi: F_4\ast \G\to U$ that restricts to $\mu$ on $F_4$ and $\psi$ on $\G$.
Consider the fibre product of $\Psi$ and $\mu$,
$$
P < (F_4\ast \G) \times F_4.
$$
Lemma \ref{ll} assures us that $P$ is finitely generated and Proposition \ref{p:PT} tells us that the
inclusion $P\hookrightarrow (F_4\ast \G) \times F_4$  induces an isomorphism of profinite completions. 

The  restriction of $\Psi$ to each conjugate of $\G$ is injective, so by the Kurosh subgroup theorem
$\ker\Psi$ is free. The projection from $P$ to the second factor of $(F_4\ast\G)\times F_4$ is onto and
has kernel $\ker\Psi$. Thus $P$ is free-by-free; more precisely, it is of the from $F_\infty\rtimes F_4$.
Define $M_\G=P$.
\qed
\begin{remark} $M_\G$ is not finitely presented.
\end{remark}
 
\section{Cohomological dimension}\label{s:cohom}

The fibre products $M_\G$ that we are considering are free-by-free (and not free) and hence have cohomological 
dimension $2$. But if $\G$ has cohomological dimension $d$ then $D(\G):=(F_r\ast \G)\times F_4$ 
has cohomological dimension $d+1$. Thus Theorem \ref{t:B} yields pairs of finitely generated, residually finite
groups that have the same profinite completion but have an arbitrary 
difference in their cohomological dimensions; for example we can take $\G\cong\Z^d$.
Moreover, $D(\Z^d)$ is good in the sense of Serre \cite{serre} and it retracts onto $\Z^{d+1}$,
so  $\wh{D(\Z^d)}\cong \wh{M}_\G$ also has cohomological dimension $d+1$.
Thus Theorem \ref{t:B} provides us with examples of groups of cohomological dimension $2$
whose profinite completions have cohomological dimension $d+1$, where $d$ is arbitrary.
One can also arrange for $ \wh{M}_\G$ to have infinite cohomological dimension (even if it
is torsion free).

\section{Torsion} \label{s:torsion}
In  \cite{lub},  Lubotzky
used the congruence subgroup property to exhibit ``as much torsion as one can wish" in the profinite
completion of certain torsion-free arithmetic groups. 
Theorem \ref{t:B} shows that torsion is similarly unconstrained in the profinite completions of free-by-free
groups, since $\wh{M}_\G\cong\wh{D(\G)}$ retracts
onto $\wh{\G}$. In this case, one can encode the torsion into $\G$ directly.

\bigskip

\noindent{\bf Acknowledgements.} I thank my longtime collaborator Alan Reid for numerous fruitful conversations about 
the structure of profinite completions, and I thank him and Khanh Le for a correspondence about the profinite invariance of local indicability that motivated me to revive a mooted sequel to  \cite{mb-gilb}. I thank  Andrei Jaikin-Zapirain for leading the organisation of the stimulating {\em Workshop on Profinite Rigidity} within the Agol Lab at ICMAT in June 2023, 
where the profinite invariance of orderability was discussed and
where I presented these results. I also thank ICMAT and its Director, Javier Aramayona, for the warmth of their hospitality.

\end{document}